\documentclass[10pt]{article}
\usepackage{amsmath, amssymb, amsfonts, amsthm, epsfig, float, graphicx, sectsty}
\usepackage[all]{xy}

\title{Constructing Directed Cayley Graphs of Small Diameter:  
A Potent Solovay-Kitaev Procedure}
\author{Henry Bradford}

\newtheorem{thm}{Theorem}[section]
\newtheorem{lem}[thm]{Lemma}
\newtheorem{propn}[thm]{Proposition}
\newtheorem{coroll}[thm]{Corollary}
\newtheorem{defn}[thm]{Definition}
\newtheorem{ex}[thm]{Example}

\newtheorem{rmrk}[thm]{Remark}

\DeclareMathOperator{\Ab}{Ab}
\DeclareMathOperator{\Aut}{Aut}
\DeclareMathOperator{\Cay}{Cay}
\DeclareMathOperator{\charac}{char}
\DeclareMathOperator{\diam}{diam}

\DeclareMathOperator{\SL}{SL}
\DeclareMathOperator{\SU}{SU}
\DeclareMathOperator{\Stab}{Stab}
\DeclareMathOperator{\Sym}{Sym}

\begin{document}

\maketitle

\begin{abstract}
Let $\Gamma$ be a group and $(\Gamma_n)_{n=1} ^{\infty}$ be a descending 
sequence of finite-index normal subgroups. 
We establish explicit upper bounds on the diameters of the 
directed Cayley graphs of the $\Gamma/\Gamma_n$, 
under some natural hypotheses on the behaviour of 
power and commutator words in $\Gamma$. 
The bounds we obtain do not depend on a choice of generating set. 
Moreover under reasonable conditions our method 
provides a fast algorithm for constructing 
directed Cayley graphs of diameter satisfying our bounds. 
The proof is closely analogous to the the 
\emph{Solovay-Kitaev procedure}, which only uses commutator words, 
but also only constructs small-diameter 
\emph{undirected} Cayley graphs. 
As an application we give directed diameter bounds 
on finite quotients of two very different groups: 
$\SL_2 (\mathbb{F}_q [[t]])$ (for $q$ even) 
and a group of automorphisms 
of the ternary rooted tree introduced by Fabrykowski and Gupta. 
\end{abstract}

\section{Introduction}

Let $G$ be a finite group, and $S \subseteq G$ be a generating set. 
We denote by $B^+ _S (n)$ 
the set of elements of $G$ expressible as positive words 
of length at most $n$ in $S$. 
The \emph{directed diameter of $G$ with respect to $S$} 
is defined to be: 
\begin{center}
$\diam^+ (G,S) 
= \min \lbrace n \in \mathbb{N} : B^+ _S (n) = G \rbrace$. 
\end{center}

The \emph{directed diameter of $G$}, 
denoted $\diam^+(G)$, 
is now defined to be the maximal value of 
$\diam^+ (G,S)$ 
as $S$ ranges over all generating subsets of $G$. 
By contrast, the \emph{(undirected) diameter of $G$ with respect 
to $S$} is $\diam (G,S) = \diam^+ (G,S \cup S^{-1})$, 
and the \emph{diameter of $G$} is the maximal value of 
$\diam (G,S)$ over $S$. 
Clearly $\diam(G) \leq \diam^+ (G)$ for any $G$. 
The purpose of this paper is to give new upper 
bounds on $\diam^+(G)$ 
for certain families of familiar finite groups, 
to provide fast algorithms for writing elements as positive 
words of length satisfying this bound, 
and to outline a procedure for proving results of this type 
in a more general setting. 

\subsection{Statement of Results}

For the sake of concision in describing 
the algorithmic aspects of our work, 
we introduce the following terminology. 

\begin{defn}
Let $(G_n)_n$ be a sequence of finite groups. 
Let $l_n ,t_n \in \mathbb{N}$ with: 
\begin{equation} \label{navigdefn}
\diam^+ (G_n) \leq l_n
\end{equation}
for all $n$. 
We say that \emph{the directed navigation problem 
for $G_n$ is solvable for the bound 
(\ref{navigdefn}) in time $t_n$} 
if there is a deterministic algorithm which, 
given an index $n$, a generating set $S_n \subseteq G_n$ 
and an element $g \in G_n$, 
outputs in time at most $t_n$ a positive word $w$ in $S_n$ 
of length at most $l_n$ which is equal to $g$ in $G_n$. 
\end{defn}

Our first result concerns congruence quotients of 
the $\mathbb{F}_q [[t]]$-analytic group 
$\SL_2 (\mathbb{F}_q [[t]])$ ($q$ even). 
In \cite{Brad} upper bounds on the (undirected) diameter 
were given for congruence quotients of many analytic 
(virtually) pro-$p$ groups, including $\SL_d (\mathbb{F}_q [[t]])$ 
for $q$ odd or $d \geq 3$. For technical reasons 
related to the structure of the associated Lie algebras, 
the case $d=2$, $q$ even fell beyond the scope of 
the methods of \cite{Brad}. 
Therefore our result here is new even for undirected diameters. 

\begin{thm} \label{SL2MainThm}
Let $\mathbb{F}_q$ be the finite field of even order $q$. 
Let $G(n,q) = \SL_2 (\mathbb{F}_q [t]/(t^n))$. 
Let $\epsilon >0$. There exist an absolute constant $C > 0$ 
such that for all $n \in \mathbb{N}$, 
\begin{equation} \label{SL2MainThmEqn}
\diam^+ \big(G(n,q)\big) 
= O_{q,\epsilon} \big( \log^{C + \epsilon} \lvert G(n,q) \rvert  \big) \text{.}
\end{equation}
Moreover there exists an absolute constant $C^{\prime} > 0$ 
such that the directed navigation problem 
for $G(n,q)$ is solvable for the bound 
(\ref{SL2MainThmEqn}) in time 
$O_{q,\epsilon} \big( \log^{C^{\prime} + \epsilon} \lvert G(n,q) \rvert  \big)$. 
\end{thm}

Our second result concerns an automorphism group of 
a regular rooted tree. 
In \cite{Brad1} the (undirected) diameters of congruence 
quotients of \emph{branch} groups acting on rooted trees 
were studied. Polylogarithmic upper bounds were 
obtained in two cases: Grigorchuk's first group and 
the Gupta-Sidki $p$-groups. 
Here we take up a different example: the group 
of automorphisms of the ternary rooted tree 
introduced by Fabrykowski and Gupta \cite{FabGup}. 
  
\begin{thm} \label{FabGupMainThm}
Let $\mathcal{T}_3$ be the ternary rooted tree. 
Let $\Gamma \leq \Aut (\mathcal{T}_3)$ 
be the Fabrykowski-Gupta group. 
Let $\Stab_{\Gamma} (n) \vartriangleleft \Gamma$ 
be the $n$th level stabiliser of $\Gamma$. 
Then there exists an absolute constant $C > 0$, 
such that for all $n \in \mathbb{N}$, 
\begin{equation} \label{FabGupMainThmEqn}
\diam^+ \big(\Gamma /\Stab_{\Gamma}(n)\big) = 
O \big( \log ^{C}\lvert\Gamma /\Stab_{\Gamma}(n) \rvert \big) \text{.}
\end{equation}
Moreover there exists an absolute constant $C^{\prime} > 0$ 
such that the directed navigation problem 
for $\Gamma /\Stab_{\Gamma}(n)$ is solvable for the bound 
(\ref{FabGupMainThmEqn}) in time 
$O \big( \log ^{C^{\prime}}\lvert\Gamma /\Stab_{\Gamma}(n) \rvert \big)$. 
\end{thm}

Once again the conclusion of Theorem \ref{FabGupMainThm} 
is new even for undirected diameters. 
$\Gamma$ and the subgroups $\Stab_{\Gamma} (n)$ 
will be defined in Section \ref{FaGuSect}. 
For now let us simply note that $\Gamma / \Stab_{\Gamma} (n)$ 
is a transitive imprimitive permutation group of degree $3^n$. 

\begin{rmrk}
The proofs of our results 
allow for the explicit computation of all constants. 
We may, for instance, take 
$C = \log(7)/\log(4/3) \approx 6.764$; $C^{\prime} = 2 + \log(4)/\log(4/3) \approx 6.819$ in 
Theorem \ref{SL2MainThm}, and take 
$C = \log(72272200)/\log(3) \approx 16.472$; 
$C^{\prime} = 1+\log(186200)/\log(3) \approx 11.054$ 
in Theorem \ref{FabGupMainThm}. 
\end{rmrk}

Theorems \ref{SL2MainThm} and \ref{FabGupMainThm} 
are both proved as consequences of our main technical result, 
Theorem \ref{PotentSKP}, 
which produces an upper bound on $\diam^+ (\Gamma/\Gamma_n)$ 
whenever $\Gamma$ is a group and $(\Gamma_n)_{n=1} ^{\infty}$ 
is a descending sequence of finite-index normal subgroups 
of $\Gamma$, such that certain properties are satisfied by 
commutators and proper powers in the $\Gamma_n$. 
It is very likely that Theorem \ref{PotentSKP}, 
or variants thereof, will also be applicable to many other groups. 

\subsection{Background and Outline of the Proof}

Estimating the diameters of finite Cayley graphs has been a subject of widespread interest for many years. Motivation comes from the problem of constructing efficient communication networks \cite{Heyd}; analysis of algorithms in computational group theory \cite{BabExp}, and various combinatorial puzzles (card-shuffling; generalizations of the Rubik's cube; the towers of Hanoi; pebble motions on graphs and so on) 
\cite{Diac,KoMiSp}. Owing to the concrete nature of these applications, one often seeks not only good diameter bounds, but also fast algorithms that express a group element as a word in a generating set, the length of which satisfies the bound. This is the navigation problem. Fortunately many of the results on diameter in permutation groups have essentially algorithmic proofs \cite{BaBeSe,BabaHaye}. 
Meanwhile the navigation problem in $\SL_d(\mathbb{F}_p)$ 
(and more generally Chevalley groups over $\mathbb{F}_p$ 
and other finite rings) was studied in \cite{Lars,Rile,KassRile}, 
where fast (sometimes probabilistic) algorithms were described and analyzed for particular generating sets 
(though a good solution to the navigation problem for groups of Lie type realizing the best known diameter bounds for arbitrary or generic generators remains elusive). 
The navigation problem is also of relevance in cryptography, 
in that efficient solutions are an obstruction to the 
construction of secure Cayley hash functions 
(see \cite{BrVdSh,PetiQuis} for a discussion). 

In spite of this impressive progress, much less is known about the directed navigation problem, as was noted in \cite{BabDir}.  This is an unsatisfactory state of affairs, as solutions to many combinatorial puzzles are better modeled by directed as opposed to undirected navigation (consider for instance the practical difficulty of inverting a large-order riffle shuffle of a deck of cards). 
Further, directed navigation is more relevant to the 
cryptanalysis of Cayley hash functions, 
in which a bit-stream is encoded as a \emph{positive} 
word in generators. 
Of the few results available, one of the most impressive is 
\cite{SchPuc}, which addresses the directed navigation 
problem for the symmetric group with respect to random pairs of generators. 
In this paper we introduce a set of tools that allow one 
to attack the directed navigation problem under certain 
group-theoretic conditions. 

The inspiration for our results comes from the 
\emph{Solovay-Kitaev procedure}. 
Given a compact metric group $\Gamma$ and a subset $S$ generating a dense subgroup, the SKP provides a framework for constructing 
a word $w$ in $S$ which approximates a given element $g \in \Gamma$ 
to a prescribed level of accuracy. 
Moreover, the length of $w$ in the word metric defined by $S$ 
is bounded in terms of the distance in $\Gamma$ between $g$ and $w$. 
The first examples to which the SKP was applied were 
the groups $\SU(k)$, 
where the problem of approximating arbitrary elements by words 
in a generating set was motivated by considerations 
coming from quantum computation \cite{DawNie}. 
The SKP has since been applied to other Lie groups 
(for instance by Dolgopyat \cite{Dolgo}, who independently 
discovered a version of the SKP and employed it to elucidate 
spectral properties of semisimple Lie groups). 
It was however also soon noticed that the similar techniques 
were relevant to finitely generated (abstract or profinite) 
groups $\Gamma$ equipped with a profinite metric, 
and that in this setting approximating elements by short words 
is equivalent to proving good diameter bounds for 
finite quotients of $\Gamma$. 
This idea has been exploited in several papers 
\cite{GamSha,Dinai,Dinai1,Brad,Brad1}. 
Moreover the SKP gives a fast solution to the navigation 
problem: this is described explicitly in 
\cite{DawNie,GamSha,Dinai1}, and 
can easily be derived from the proofs of the results in 
\cite{Dinai,Brad,Brad1}. 

How does the SKP work? We assume that there is a neighbourhood $U$ of the identity in $\Gamma$ satisfying two hypotheses. 
The first hypothesis that every element $z$ of $U$ 
lying sufficiently close to the identity is approximable 
by a product of (a bounded number of) commutators $[x_i,y_i]$, 
where $x_i,y_i \in U$ are significantly further from the identity 
than $z$ is. The second, complementary, hypothesis is that for 
$x,y \in U$, the commutator $[x,y]$ is significantly closer 
to the identity than $x$ and $y$. 
It follows from the latter that if the pairs $(x,y)$ and $(\tilde{x},\tilde{y})$ 
are close, then $[x,y]$ and $[\tilde{x},\tilde{y}]$ 
are even closer. 
If $z \in U$ is the error in our existing verbal approximation $\tilde{g}$ to $g \in \Gamma$;  
$[x_1,y_1] \cdots [x_A,y_A]$ is an approximation to $z$ 
(which exists by the first hypothesis) 
and $\tilde{x}_i,\tilde{y}_i$ are verbal approximations to $x_i,y_i$ (which we may assume exist by induction), then  $\tilde{g}[\tilde{x}_1,\tilde{y}_1]\cdots[\tilde{x}_A,\tilde{y}_A]$ is a better verbal approximation to $g$. 

In the present paper we modify this strategy, 
in that we replace the first hypothesis by the requirement that 
$z$ is approximable by a product of $k$th powers $y_i ^k$, 
for $y_i \in U$ and $k \geq 2$ fixed. 
To implement the induction step, 
we must then also strengthen the second hypothesis, 
by requiring that taking $k$th powers moves elements of $U$ 
closer to the identity, as well as commutators. 
As we shall see below (Remark \ref{dimserex}), 
a very natural setting in which the second hypothesis holds is 
when $k=p$ is a prime and $\Gamma$ 
is a residually $p$-finite group, 
equipped with the profinite metric defined by the mod-$p$ dimension series. 
Because it relies heavily on properties of proper powers, 
it seems appropriate to term the new method a \emph{potent} 
Solovay-Kitaev procedure. 
The fact that it yields a \emph{directed} diameter bound 
follows from the fact that the proper powers used to express elements close to the identity are \emph{positive words}. 

A version of the SKP was also used 
by Bourgain and Gamburd \cite{BoGa1,BoGa2} 
(in conjunction with other tools) 
to produce new examples of \emph{expander Cayley graphs}. 
Expanders are sparse finite regular graphs 
with very strong connectivity and mixing properties. 
For instance they have logarithmic (undirected) diameter 
and, which is more, the endpoints of paths of logarithmic 
length are \emph{equidistributed} over the graph. 
Expanders have remarkable and diverse applications 
across pure mathematics, communication theory and theoretical 
computer science; we refer the reader to the excellent 
survey articles \cite{HoLiWi,LubSurv} for an overview of these. 
It would be very interesting to investigate the possibility 
of adapting the \emph{potent} SKP to 
construct new examples of expanders. 

In spite of the obvious analogies between the original 
SKP and our new potent variant, and the relevance 
of the former to approximation problems in real and 
complex Lie groups, the potent SKP 
appears to be predominantly a ``non-analytic'' 
phenomenon: raising elements of a real or complex Lie group 
to a proper power does not generically move them closer to 
the identity. 
Indeed the problem of approximating an arbitrary 
element in a Lie group by a short positive word in an arbitrary 
generating set appears to be open. 
As noted in \cite{DawNie}, a solution to this problem 
for $\SU(d)$ would be of interest in the context of 
quantum computation: the hypothesis of a symmetric 
generating set, although group-theoretically natural, 
has no clear justification when the set of generating matrices 
is interpreted as the instruction set of a quantum computer. 
The potent SKP \emph{does} yield directed diameter bounds for 
quotients of $p$-adic analytic groups, by exploiting their connection with \emph{powerful pro-$p$ groups}, 
but the diameter bounds are rather weak: 
for instance for the groups $G(d,p,n) = \SL_d(\mathbb{Z}/p^n\mathbb{Z})$
we would obtain $\diam^+ \big( G(d,p,n) \big)
 = O_{p,d} \big(\lvert G(d,p,n) \rvert^{1/(d^2-1)} \big)$, 
which compares poorly with the polylogarithmic 
undirected diameter bounds for these groups in \cite{Brad}. 
We discuss the relevance of the potent SKP to $p$-adic 
analytic groups futher in Section \ref{padicsect}. 

As noted above, for any finite group $G$ we have 
$\diam(G) \leq \diam^+ (G)$. 
Somewhat surprisingly, 
there is also a converse inequality due to Babai. 

\begin{thm}[\cite{BabDir} Corollary 2.3] \label{BabaiDirectedThm}
Let $G$ be a finite group. Then: 
\begin{center}
$\diam^+ (G) = O \big( \diam(G) \log\lvert G\rvert ^2 \big)$. 
\end{center}
\end{thm}

As a result, all groups with polylogarithmic diameter 
also have polylogarithmic directed diameter, 
and where the degree of the polylogarithm in the former is explicitly known, so is that in the latter. 
In spite of this, there are advantages to 
deriving directed diameter bounds without the use of 
Theorem \ref{BabaiDirectedThm}, 
even when good (undirected) diameter bounds are known. 
In particular, the proof of Theorem \ref{BabaiDirectedThm} 
is non-constructive, 
so does not yield any non-trivial solution to the 
directed navigation problem 
(see \cite{BabDir} Section 5 for a discussion of this 
and related problems). 
One can be very confident that the potent SKP 
(either in the form of Theorem \ref{PotentSKP} or with modifications)
will provide solutions to the directed navigation 
problem for many of the other 
$\mathbb{F}_q [[t]]$-analytic groups; 
branch pro-$p$ groups, and Nottingham groups of finite fields
studied in \cite{Brad,Brad1}. 
These solutions will moreover witness directed diameter bounds 
qualitatively similar, if quantitatively weaker, 
than those obtained by combining Theorem \ref{BabaiDirectedThm} 
with the results of \cite{Brad,Brad1}. 
Nevertheless, owing to the availability of Theorem 
\ref{BabaiDirectedThm}, we have opted predominantly to 
illustrate the implementation of the potent SKP 
with examples for which polylogarithmic 
\emph{undirected} diameter bounds were not previously known. 

The paper is structured as follows. In Section \ref{PSKSect} 
we develop the potent Solovay Kitaev procedure in an abstract 
setting, giving sufficient conditions on the behaviour 
of power and commutator words in the sequence $(\Gamma_i)_i$ for 
a good upper bound on the $\diam^+(\Gamma/\Gamma_i)$ to hold. 
In Sections \ref{SL2Sect} and 
\ref{FaGuSect} we prove, respectively, 
Theorems \ref{SL2MainThm} and \ref{FabGupMainThm}. 
In Section \ref{padicsect} we derive from the potent SKP a weak 
upper bound on directed diameters in quotients of $p$-adic 
analytic groups. 
In Section \ref{mixsect} we discuss some 
implications of our results for spectral gaps and mixing 
times of random walks. 

\section{The Procedure} \label{PSKSect}

In this Section we describe the potent Solovay-Kitaev Procedure 
in an abstract group-theoretic context. 
The Procedure is expressed in Theorem \ref{PotentSKP}. 
Sections \ref{SL2Sect} and \ref{FaGuSect} will then 
be devoted to proving that the hypotheses of Theorem 
\ref{PotentSKP} hold in the relevant settings such that 
Theorems \ref{SL2MainThm} and \ref{FabGupMainThm} follow 
immediately. 

We start with an observation to the effect that, 
given an approximation to a group element, 
the $k$th power of the element is well-approximated 
by the $k$th power of the approximation. 
For $N \leq G$ denote by $\mho_k (N)$ the subgroup 
of $G$ generated by all $k$th powers of elements 
of $N$. Note that $\mho_k (N)$ is normal in $G$ 
whenever $N$ is. 

\begin{lem} \label{HPLem}
Let $\Gamma$ be a group; 
let $M , N \vartriangleleft \Gamma$ and let 
$k \in \mathbb{N}_{\geq 2}$. 
Then for all $g \in M , h \in N$, 
\begin{center}
$(gh)^k g^{-k} \in [M,N] \mho_k (N)$. 
\end{center}
\end{lem}

\begin{proof}
Let $\langle [g,h] \rangle^{\Gamma}$ be the normal closure of 
$[g,h]$ in $\Gamma$. Then $\langle [g,h] \rangle^{\Gamma} \leq [M,N]$ 
(since $M,N \vartriangleleft \Gamma$), and 
$(gh)^k h^{-k} g^{-k}$ is clearly trivial in 
$\Gamma/\langle [g,h] \rangle^{\Gamma}$ 
(since the images of $g$ and $h$ in the latter quotient commute). 
Thus $(gh)^k h^{-k} g^{-k} \in [M,N]$ and the result follows. 
\end{proof}

The conclusion of Lemma \ref{HPLem} will be useful 
in situations where $[M,N] \mho_k (N)$ is much smaller than $N$. 

\begin{ex} \label{dimserex}
Let $(\Gamma_n)_{n=1} ^{\infty}$ be a descending sequence of finite-index normal subgroups of $\Gamma$. 
Suppose that for all $m,n \in \mathbb{N}$: 
\begin{itemize}
	\item[(i)]  $[\Gamma_n , \Gamma_m] \subseteq \Gamma_{n+m}$;   
	\item[(ii)] $\mho_k (\Gamma_n) \subseteq \Gamma_{kn}$. 
\end{itemize}
Let $n \leq m$ and let $g \in \Gamma_n$, $h \in \Gamma_m$. 
Then by Lemma \ref{HPLem}: 
\begin{center}
	$(gh)^k \equiv g^k \mod \Gamma_{n+m}$. 
\end{center}
It is a classical fact that $(\Gamma_n)_{n=1} ^{\infty}$ 
satisfies conditions 
(i) and (ii) above with $k=p$ a prime when $\Gamma_n$ is the 
\emph{mod-$p$ dimension series} of $\Gamma$. 
Recall that the latter is the sequence 
$(D_n (\Gamma))_{n=1} ^{\infty}$ of normal subgroups of $\Gamma$ given by:
\begin{center}
	$D_n (\Gamma)= \lbrace g \in \Gamma : g-e \in I^n \rbrace$
\end{center}
where $I$ is the \emph{augmentation ideal} of the group algebra $\mathbb{F}_p \Gamma$, 
defined to be the kernel of the \emph{augmentation mapping} $\phi : \mathbb{F}_p \Gamma \rightarrow \mathbb{F}_p$, 
which is given by: 
\begin{center}
	$\phi (\sum^{\prime} \lambda_g \cdot g) = \sum^{\prime} \lambda_g$. 
\end{center}
Alternatively, $D_n (\Gamma)$ may be defined recursively by 
$D_1 (\Gamma)= \Gamma$, $D_{n+1}(\Gamma) = [\Gamma,D_n (\Gamma)] \mho_p (D_{\lceil (n+1)/p \rceil} (\Gamma))$. 

Another example of a sequence $(\Gamma_n)_{n=1} ^{\infty}$ 
in which conditions (i) and (ii) above hold, 
and which will be relevant to Theorem \ref{SL2MainThm}, 
is given below (see Lemma \ref{SL2Lem2}). 
\end{ex}

\begin{thm} \label{PotentSKP}
Let $(M_n)_{n=1} ^{\infty}$, $(N_n)_{n=1} ^{\infty}$ 
be sequences of finite-index normal subgroups in $\Gamma$. 
Let $(A_n)_{n=1} ^{\infty},(k_n)_{n=1} ^{\infty}$ be a sequence of positive integers. 
Suppose that for all $n \in \mathbb{N}$: 
\begin{itemize}
\item[(i)] $N_n \leq M_n$; 
\item[(ii)] $[M_n,N_n] \leq N_{n+1}$; 
\item[(iii)] $\mho_{k_n} (N_n) \leq N_{n+1}$; 
\item[(iv)] For all $z \in N_n$, 
there exist $y_1 , \ldots , y_{A_n} \in M_n$ such that: 
\begin{equation} \label{potentapproxeqn}
y_1 ^{k_n} \cdots y_{A_n} ^{k_n} z^{-1} \in N_{n+1}\text{.} 
\end{equation}
\end{itemize}
Then for all $n \in \mathbb{N}$: 
\begin{equation} \label{PotentSKDiamEqn}
\diam^+ (\Gamma/N_n) 
\leq l_n = \lvert \Gamma:N_1 \rvert \prod_{i=1} ^{n-1} (1+A_i k_i)\text{.}
\end{equation}
Further suppose that for all $m \in \mathbb{N}$, 
the times needed to compute: 
\begin{itemize}
\item[(a)] The product $gh$ 
of given input elements $g,h \in \Gamma/N_m$; 
\item[(b)] The inverse $g^{-1}$ 
of a given input element $g \in \Gamma/N_m$; 
\item[(c)] $N_m y_1 , \ldots , N_m y_{A_n}$, 
given input $1 \leq n \leq m$ and $N_m z$, 
where $y_i \in M_n$ and $z \in N_n$ are as in 
(\ref{potentapproxeqn}) 
\end{itemize}
are at most $f(m)$. 
Then 
the directed navigation problem for $\Gamma/N_n$ is 
solvable for the bound (\ref{PotentSKDiamEqn}) in time: 
\begin{equation} \label{runtimebound}
f(n) \Big( C \lvert S \rvert^{\lvert \Gamma:N_1 \rvert + 1}  \prod_{i=1} ^{n-1} (A_i+1) + \sum_{i=1} ^{n-1} (A_i k_i + 3)
\prod_{j=i} ^{n-2} (A_j + 1) \Big)
\end{equation}
for $C>0$ an absolute constant. 
\end{thm}

\begin{proof}
First let us establish the diameter bound. 
For $n=1$ the conclusion is trivial. 
Suppose by induction that $\diam^+ (\Gamma/N_n) \leq l_n$. 
Let $S_{n+1} \subseteq \Gamma/N_{n+1}$ be a generating set and  
let $S_n$ be the image of $S_{n+1}$ in $\Gamma/N_n$. 
Then $S_n$ generates $\Gamma/N_n$. Let $g \in \Gamma/N_{n+1}$. 
By inductive hypothesis there exists $w \in B_{S_{n+1}} ^+ (l_n)$ 
such that $z = w^{-1} g \in N_n$. By hypothesis (iv) 
there exist $y_1 , \ldots , y_{A_n} \in M_n$ such that 
$y_1 ^{k_n} \cdots y_{A_n} ^{k_n} z^{-1} \in N_{n+1}$. 

By inductive hypothesis there exist, for $1 \leq i \leq A_n$, 
$\tilde{y}_i \in B_{S_{n+1}} ^+ (l_n)$ 
such that $y_i \tilde{y}_i ^{-1} \in N_n$. 
Combining hypotheses (ii) and (iii) with Lemma \ref{HPLem}, 
we have $y_i ^{k_n} (\tilde{y}_i) ^{-k_n} \in N_{n+1}$. 
Then: 
\begin{center}
$g = wz 
\equiv w (\tilde{y}_1)^{k_n} \cdots (\tilde{y}_{A_n})^{k_n} 
\mod N_{n+1}$
\end{center}
and $w (\tilde{y}_1)^{k_1} \cdots (\tilde{y}_{A_n})^{k_n} 
\in B_{S_{n+1}} ^+ (l_n(1+A_n k_n))$. 
The diameter bound follows by induction. 

We now describe and analyze an algorithm 
$\mathtt{APPROX}(n,i,g,S)$, 
which takes as input $n,i \in \mathbb{N}$ with $i \leq n$, 
$g \in \Gamma/N_n$ and $S \subseteq \Gamma/N_n$, and outputs 
both a positive word $\tilde{w} \in F(S)$ of length at most $l_i$ 
and the evaluation $w$ of $\tilde{w}$ in $\Gamma/N_n$, 
with the property that $g \equiv w \mod N_i$. 
The algorithm required by the statement of the Theorem 
will be $\mathtt{APPROX}(n,n,g,S)$. 

First note that $\mathtt{APPROX}(n,1,g,S)$ 
runs in time $O(\lvert S \rvert^{\lvert \Gamma:N_1 \rvert + 1}f(n))$: we may simply compute all products 
of elements in $S$ of length at most $\lvert \Gamma:N_1 \rvert$; 
one of these will agree with $g$ modulo $N_1$. 

Now we employ recursion. 
Given $1 \leq i \leq n-1$, let $(\tilde{w}_i,w_i)$ be the output of 
$\mathtt{APPROX}(n,i,g,S)$. 
Then $z = w_i ^{-1} g \in N_i$. 
Compute $y_1 , \ldots , y_{A_i} \in M_i$ 
as in (\ref{potentapproxeqn}); as hypothesized in (c) above, 
this requires time at most $f(n)$. 

Let $(\tilde{v}_{i,j},v_{i,j})$ 
be the output of $\mathtt{APPROX}(n,i,y_j,S)$. 
The output of $\mathtt{APPROX}(n,i+1,g,S)$ is 
$(\tilde{w}_{i+1},w_{i+1})$, where 
$\tilde{w}_{i+1} = \tilde{w}_i \tilde{v}_{i,1} ^{k_i} \cdots
\tilde{v}_{i,A_i} ^{k_i}$ and 
$w_{i+1} = w_i v_{i,1} ^{k_i} \cdots v_{i,A_i} ^{k_i}$. Our proof of the diameter bound above witnesses 
that $\tilde{w}_{i+1},w_{i+1}$ have the required properties. 

Finally take $t_{n,i} \in \mathbb{N}$ such that 
$\mathtt{APPROX}(n,i,g,S)$ runs in time at most $t_{n,i}$ 
for all $g,S$. As noted above, we may take: 
\begin{center}
$t_{n,1} = C \lvert S \rvert^{\lvert \Gamma:N_1 \rvert + 1}f(n)$
\end{center}
For $1 \leq i \leq n-1$ note that to implement $\mathtt{APPROX}(n,i+1,g,S)$ 
we must call $\mathtt{APPROX}(n,i,h,S)$ 
for $A_i+1$ elements $h$, 
and carry out $A_i k_i +1$ computations of type (a) 
and one each of type (b) and (c). We may therefore take: 
\begin{center}
$t_{n,i} = (A_i +1)t_{n,i} + (A_i k_i +3)$
\end{center}
and the conclusion (\ref{runtimebound}) follows. 
\end{proof}

\begin{rmrk} \label{improvementsremark}
\begin{itemize}
\item[(i)] The statement of Theorem \ref{PotentSKP} is more general than 
we shall need in the setting of Theorems \ref{SL2MainThm} 
and \ref{FabGupMainThm}, where $(k_n)$ will be a constant 
sequence, and $(A_n)$ will be periodic. 
We state Theorem \ref{PotentSKP} in this general form to 
emphasize the adaptability of the potent SKP, 
and its potential applicability to problems 
much more diverse than the applications we give here. 

\item[(ii)] Equally, additional refinements 
to Theorem \ref{PotentSKP} are possible, 
which improve the diameter bounds 
and the runtime of our algorithm. 
For instance, suppose there exists a constant 
$n_0 \in \mathbb{N}$ such that for all $n$, $M_{n+n_0} \leq N_n$. 
Then for any generating set $S \subseteq \Gamma/N_n$ 
and any $1 \leq i \leq n-1$, 
we have  $N_i / N_{i+1} \subseteq B^+ _S (L_i) N_{i+1} /N_{i+1}$, 
where $L_0 = \lvert \Gamma:N_1\rvert$, 
and $L_i = A_i k_i (L_{i-n_0} + \cdots + L_{i-1})$ 
for $i \geq 1$ (with $L_i=0$ for negative indices). Thus: 
\begin{equation} \label{improvermrkeqn}
\diam^+ (\Gamma/N_n) \leq L_0 +\cdots + L_{n-1}\text{.}
\end{equation}
To see that this is a stronger upper bound, 
note that the bound (\ref{PotentSKDiamEqn}) may be expressed as 
$l_n = L^{\prime} _0 + \cdots L^{\prime} _{n-1}$, 
where $L^{\prime} _0 = \lvert \Gamma:N_1\rvert$ 
and $L^{\prime} _i = A_i k_i (L^{\prime} _0 + \cdots + L^{\prime} _{i-1})$ 
for $i \geq 1$. 

\item[(iii)] The initial step of our induction, 
which yields the trivial bounds 
$\diam^+ (\Gamma/N_1) \leq \lvert \Gamma:N_1 \rvert$ 
and a solution to the directed navigation problem 
for $\Gamma/N_1$ in time 
$O(\lvert S \rvert^{\lvert \Gamma:N_1 \rvert + 1})$, 
is far from optimal in many cases. 
For instance $\diam^+ (\SL_2 (q)) = O(\log(q)^c)$ 
for an absolute constant $c$ \cite{Dinai0}, 
which enables improvements to the constants appearing 
in our Theorem \ref{SL2MainThm}. 
\end{itemize}
\end{rmrk}

\section{Proofs for $\SL_2 (\mathbb{F}_q [[t]])$} \label{SL2Sect}

Let $\mathbb{F}_q$ be a finite field of even order $q$, 
let $\mathbb{F}_q [[t]]$ be the power series ring of $\mathbb{F}_q$ 
and let $\Gamma = \SL_2 (\mathbb{F}_q [[t]])$. 
For $n \in \mathbb{N}$, let: 
\begin{center}
$K_n = \Gamma \cap (I_2 + t^n \mathbb{M}_2 (\mathbb{F}_q [[t]])) 
= \ker (\pi_n)$,
\end{center}
where $\pi_n : \Gamma \twoheadrightarrow 
\SL_2 (\mathbb{F}_q [t]/(t^n))$ is the congruence map. 
Hence $(K_n)_n$ is a descending chain of finite-index normal subgroups of $\Gamma$. 

\begin{lem} \label{SL2Lem1}
Let $n,m \in \mathbb{N}$. Then: 
\begin{itemize}
\item[(i)] $[K_n,K_m] \subseteq K_{n+m}$; 

\item[(ii)] $\mho_2 (K_n) \subseteq K_{2n}$. 

\end{itemize}
\end{lem}

\begin{proof}
Let $X,\tilde{X},Y,\tilde{Y} \in \mathbb{M}_2 (\mathbb{F}_q [[t]])$ 
be such that $g = I_2 + t^n X$, $g^{-1} = I_2 + t^n \tilde{X}$, 
$h = I_2 + t^m Y$, $h^{-1} = I_2 + t^m \tilde{Y}$. 
\begin{itemize}
\item[(i)] Since $ g^{-1} \cdot g = h^{-1} \cdot h = I_2$, 
\begin{equation} \label{SL2equivinv}
X+\tilde{X}+t^n \tilde{X}X = Y+\tilde{Y}+t^m \tilde{Y}Y = 0\text{.}
\end{equation}
Thus:
\begin{align*}
[g,h] & = (I_2 + t^n \tilde{X})(I_2 + t^m \tilde{Y})
(I_2 + t^n X)(I_2 + t^m Y)\\
& \equiv I_2 + t^n (X+\tilde{X}) + t^m (Y+\tilde{Y}) 
+ t^{2n} \tilde{X}X + t^{2m} \tilde{Y}Y \\
& \equiv I_2 \mod t^{n+m} \text{    (by (\ref{SL2equivinv}))}
\end{align*}
so $[g,h] \in K_{n+m}$. 

\item[(ii)] $g^2 = (I_2 + t^n X)^2 = I_2 + t^{2n} X^2 \in K_{2n}$ (since $\charac(\mathbb{F}_q)=2$). 

\end{itemize}
\end{proof}

\begin{lem} \label{SL2Lem2}
Let $z \in K_{3n}$. Then there exist $y_1 , y_2 , y_3 \in K_n$ 
such that:
\begin{center}
$y_1 ^2 \cdot y_2 ^2 \cdot y_3 ^2 \cdot z^{-1} \in K_{4n}$. 
\end{center}
\end{lem}

\begin{proof}
For $\alpha \in \mathbb{F}_q [[t]]$ 
define the following elements of $K_n$: 
\begin{center}
$D_n (\alpha) = \left(
\begin{array}{cc}
1+t^{2n}\alpha & t^n \\
t^n \alpha & 1
\end{array}
\right)$, \\
$E_n (\alpha) = \left(
\begin{array}{cc}
1+t^n & t^n \alpha \\
0 & (1+t^n)^{-1}
\end{array}
\right)$, \\
$F_n (\alpha) = \left(
\begin{array}{cc}
(1+t^n)^{-1} & 0 \\
t^n \alpha & 1+t^n
\end{array}
\right)$. 
\end{center}
Now consider $z \in K_{3n}$ There exist $a,b,c,d \in \mathbb{F}_q [[t]]$ 
such that: 
\begin{center}
$z = I_2 + 
t^{3n} \left(
\begin{array}{cc}
a & b \\
c & d
\end{array}
\right)$
\end{center}
Then $1 = \det(g) = 1 + t^{3n} (a+d) + t^{6n} (ad-bc)$, 
so $a \equiv d \mod t^{3n}$. Set: 
\begin{center}
$y_1 = D_n (t^n \overline{a})$, $y_2 = E_n (\overline{b})$, 
$y_3 = F_n (\overline{c}) \in K_n$. 
\end{center} 
(for any $\overline{a} \equiv a,
\overline{b} \equiv b,
\overline{c} \equiv c \mod t^n$). We compute: 
\begin{center}
$y_1 ^2 \cdot y_2 ^2 \cdot y_3 ^2 
\equiv \left(
\begin{array}{cc}
1+ t^{3n}a & t^{3n}b \\
t^{3n} c & 1 + t^{3n}a
\end{array}
\right)
\equiv z \mod t^{4n}$
\end{center}
as required. 
\end{proof}

\begin{rmrk} \label{SL2approxtimermrk}
It is clear from the proof of Lemma \ref{SL2Lem2} that there 
is an algorithm which, given $z \in K_{3n}$, 
computes the $y_1 , y_2 , y_3$ in time $O(n)$ 
(by reading the coefficients $a,b,c$ modulo $t^n$ 
in our expression for $z$ and substituting 
into our expressions for $y_1 , y_2 , y_3$). 
\end{rmrk}

\begin{proof}[Proof of Theorem \ref{SL2MainThm}]
Let $(\alpha_n)_n$, $(\beta_n)_n$ 
be ascending sequences of integers such that 
(a) $\alpha_n + \beta_n \geq \beta_{n+1}$ 
and (b) $\beta_n \geq 3 \alpha_n$. 
Note that (a) and (b) together imply 
(c) $4 \beta_n / 3 \geq \beta_{n+1}$. 

We define $M_n = K_{\alpha_n},N_n = K_{\beta_n} 
\leq \Gamma$ 
and set $A_n = 3$, $K_n = 2$. We check that these sequences satisfy the hypotheses of Theorem \ref{PotentSKP}. 
Hypothesis (i) is clear; 
hypotheses (ii) and (iii) follow from Lemma \ref{SL2Lem1} 
and the above conditions, 
and hypothesis (iv) follows from Lemma \ref{SL2Lem2} 
and condition (c) above. 

We therefore have: 
\begin{align*}
\diam^+ (\Gamma/N_n) & \leq 7^{n-1} \lvert \Gamma:N_1 \rvert \\
& = O_{\beta_1,q} \big( \log \lvert \Gamma:N_n \rvert^{n\log(7)/\log(\beta_n)}\big)
\end{align*} 
(since $\lvert SL_2 (\mathbb{F}_q [t]/(t^m)) \rvert
= (q^2-1)q^{3m-2}$). 
The bound (\ref{SL2MainThmEqn}) for this 
subsequence of $G(n,q)=\Gamma/K_n$ follows from the 
easy observation that for all $\epsilon > 0$ 
we may take $\alpha_n , \beta_n 
= \Omega_{\epsilon}((\frac{4}{3} - \epsilon)^n)$. 

For the directed navigation problem, 
we observe that multiplying two elements of 
$\SL_2(\mathbb{F}_q [t]/(t^n))$ involves $O(n^2)$ 
multiplications and additions of pairs of elements of 
$\mathbb{F}_q$, so may be achieved in time $O_q(n^2)$. 
Inversion involves only the rearrangement of co-ordinates 
so may be accomplished in linear time, 
as may computing the approximations $y_i$ to a given $z$ 
(by Remark \ref{SL2approxtimermrk}). 
We therefore satisfy conditions (a), (b) and (c) 
of Theorem \ref{PotentSKP} with $f(n)=\beta_n ^2$, 
so from (\ref{runtimebound}), 
we have a solution in time: 
\begin{align*}
O \big(\beta_n ^2 4^n \lvert S \rvert^{\lvert G:K_{\beta_1} \rvert+1}\big)
& = O \big(\lvert S \rvert^{O_{q,\epsilon}(1)} 
\log \lvert \Gamma:N_n \rvert^{2+ \frac{\log(4)}{\log(4/3 - \epsilon)}}\big)\text{.}
\end{align*}
The conclusions of Theorem \ref{SL2MainThm} 
for general $G(n,q)=\Gamma/K_n$ follow from the 
above bounds for $\Gamma/N_m$ by taking $m$ such that 
$N_m \leq K_n \leq N_{m-1}$ and comparing the 
indices of $N_m$ and $K_n$ in $\Gamma$. 
\end{proof}

\begin{rmrk} \normalfont
Set $\Gamma = \SL_d (\mathbb{F}_q[[t]])$ for $d \geq 3$, $q$ even, 
and again take: 
\begin{center}
$K_n = \Gamma \cap (I_2 + t^n \mathbb{M}_2 (\mathbb{F}_q [[t]]))$. 
\end{center}
Slightly modifying the above construction for $\SL_2$, 
it is easy to show that every element of $K_{3n}$ 
may be written modulo $K_{4n}$ as the product of four squares 
of elements in $K_n$ (provided $q$ is sufficiently large, 
depending on $d$). 
It follows that $\diam^+ (\Gamma/K_n) 
= O_{d,q,\epsilon} (\log \lvert \Gamma:K_n \rvert^{C+\epsilon})$ 
for all $\epsilon > 0$, where 
$C = \log(9)/\log(4/3) \approx 7.638$. 
For comparison, the results of \cite{Brad} yield 
$\diam(\Gamma/K_n) = O_{d,q}(\log \lvert \Gamma:K_n \rvert^{C^{\prime}})$, 
where $C^{\prime} = \log(44)/\log(2) \approx 5.459$. 
Thus the bound for $\diam^+$ obtained by combining the 
latter bound for $\diam$ with Theorem \ref{BabaiDirectedThm} 
is asymptotically very slightly better than that obtained 
by applying the potent Solovay-Kitaev procedure directly, 
but does not provide a solution to the directed navigation problem, 
which the potent SKP does. 
\end{rmrk}

\section{The Fabrykowski-Gupta Group} \label{FaGuSect}

Throughout this Section we denote the $n$-fold 
Cartesian product of the set $X$ by $X^{(\times n)}$, 
to avoid possible confusion with 
the $n$-fold product of a subset of a group or monoid. 
For $m \geq 2$ define the \emph{$m$-ary rooted tree} 
to be the graph $\mathcal{T}_{\mathcal{A}}$ 
with vertex set $\mathcal{A} ^*$ 
the set of formal positive words on alphabet $\mathcal{A}$, 
a set of cardinality $m$, 
and edges $(w,wa)$ for $w \in \mathcal{A} ^*$ and 
$a \in \mathcal{A}$. 
The set $\mathcal{A}^n$ of words of length $n$ in $\mathcal{A}$ 
(that is, the set of vertices of $\mathcal{T}_{\mathcal{A}}$ 
at distance $n$ from the \emph{root vertex}, represented by the 
empty word) is known as the \emph{$n$th level set} of $\mathcal{T}_{\mathcal{A}}$. 

The group $\Aut (\mathcal{T}_{\mathcal{A}})$ 
of graph automorphisms of $\mathcal{T}_{\mathcal{A}}$ 
is precisely the set of permutations of $\mathcal{A} ^*$ 
which respect prefixes, and in particular fixes the root vertex. 
The kernel of the action of $\Aut (\mathcal{T}_{\mathcal{A}})$ on the $n$th level set $\mathcal{A}^{n}$ 
will be called the \emph{$n$th level stabiliser} and denoted $\Stab (n)$; 
it is naturally isomorphic to $\Aut (\mathcal{T}_{\mathcal{A}})^{(\times \lvert \mathcal{A} \rvert^n)}$. 
If $\Gamma \leq \Aut (\mathcal{T}_{\mathcal{A}})$ we write $\Stab_{\Gamma} (n)$ for $\Gamma \cap \Stab (n)$. 

For any $\phi \in \Aut (\mathcal{T}_{\mathcal{A}})$, 
there exists a unique $\sigma_{\phi} \in \Sym (\mathcal{A})$ 
such that for any $x \in \mathcal{A}$, there exists a unique $\phi_x \in \Aut (\mathcal{T}_{\mathcal{A}})$ such that: 
\begin{center}
$\phi (x w) = \sigma_{\phi} (x) \phi_x (w)$, for all $w \in \mathcal{A}^*$. 
\end{center}
The induced map $\psi : \phi \mapsto (\phi_x)_{x \in \mathcal{A}} \cdot \sigma_{\phi}$ 
gives an isomorphism \linebreak$\Aut (\mathcal{T}_{\mathcal{A}}) 
\rightarrow \Aut (\mathcal{T}_{\mathcal{A}}) \wr \Sym (\mathcal{A})$. 
Note that the level stabilisers may be described recursively by 
$\Stab (0) = \Aut (\mathcal{T}_{\mathcal{A}})$ and 
$\Stab (n+1) = \psi^{-1} (\Stab (n)^{(\times \lvert \mathcal{A} \rvert)})$. 

Of particular interest among the subgroups of $\Aut (\mathcal{T}_{\mathcal{A}})$ 
are those whose action on $\mathcal{T}_{\mathcal{A}}$ 
is \emph{branch}. 
Our characterization of such groups is based on that appearing in \cite{Bar}. 

\begin{defn} \label{branchdefn}
Let $\Gamma \leq \Aut (\mathcal{T}_{\mathcal{A}})$. 
$\Gamma$ is \emph{(regular) branch} if: 
\begin{itemize}
\item[(i)] The action of $\Gamma$ on $\mathcal{A}$ is transitive; 
\item[(ii)] $\psi (\Stab_{\Gamma} (1)) \leq \Gamma^{(\times \lvert \mathcal{A} \rvert)}$; 
\item[(iii)] $\Gamma$ has a finite-index subgroup $K$ such that $K^{(\times \lvert \mathcal{A} \rvert)} \leq \psi(K)$. 
\end{itemize}
We will simply say that a group $\Gamma$ \emph{branches over $K$} when the alphabet $\mathcal{A}$ 
and the action of $\Gamma$ on $\mathcal{A}^*$ is clear. 
\end{defn}

Henceforth we usually suppress the map $\psi$ from expressions and identify subgroups of $\Gamma$ 
with their image under $\psi$, 
so we may for instance speak of 
$K^{(\times \lvert \mathcal{A} \rvert)}$ as a subgroup of $K$; 
$\Stab_{\Gamma} (n)$ as a subgroup of $\Gamma ^{(\times \lvert \mathcal{A} \rvert^n)}$ and so on. 

We are now ready to define $\Gamma$. 
Let $\mathcal{A} = \lbrace 0,1,2 \rbrace$ and write 
$\mathcal{T}_{\mathcal{A}} = \mathcal{T}_3$. 
The \emph{Fabrykowski-Gupta} group is the subgroup $\Gamma$ 
of $\Aut (\mathcal{T}_3)$ which is generated by the two 
automorphisms $a,b$ defined by: 
\begin{center}
$a(0w)=1w, a(1w)=2w, a(2w)=0w$, \\
$b(0w)=0(aw), b(1w)=1w, b(2w)=2(bw)$. 
\end{center}
That is, $a$ cyclically permutes the subtrees rooted at $0$, $1$ 
and $2$, while $b \in \Stab_{\Gamma} (1)$ is defined recursively 
by $b=(a,1,b)$. It is easily seen that $a$ and $b$ have order $3$. 

Let $K = [\Gamma,\Gamma]$ be the derived subgroup of $\Gamma$. 
We have $K \leq \Stab_{\Gamma} (1)$, since 
$\Gamma / \Stab_{\Gamma} (1) \cong C_3$ is abelian. 

Consider the following elements of $K$: 
\begin{center}
$x_1 = [a,b] = (b^{-1}a,a^{-1},b)$\\
$x_2 = [a,x_1] = (ba, a^{-1}ba^{-1},ab)$. 
\end{center}

\begin{propn} \label{branchingpropn}
\begin{itemize}
\item[(i)] $\Gamma$ branches over $K$; 

\item[(ii)] $\Gamma / K \cong C_3 \times C_3$, 
with basis $Ka,Kb$; 

\item[(iii)] $K/K^{(\times 3)} \cong C_3 \times C_3$, 
with basis $K^{(\times 3)} x_1,K^{(\times 3)} x_2$. 

\end{itemize}
\end{propn}

\begin{proof}
(i) is proved as Proposition 6.2 in \cite{BartGrig}. 
(ii) and (iii) also follow easily from the results of 
\cite{BartGrig} Section 6, however for the sake of completeness 
we give a self-contained proof. 

For (ii), $\Gamma / K$ is certainly a quotient of $C_3 \times C_3$, 
since $\Gamma$ is generated by two elements of order $3$. 
On the other hand, there is a natural homomorphism 
$\Gamma \rightarrow C_3 \wr C_3$ (with kernel $\Stab_{\Gamma} (2)$). 
Inspection of the action of $a$ and $b$ on $\mathcal{T}_3$ 
confirms that this homomorphism is surjective. 
But $(C_3 \wr C_3)^{\Ab} \cong C_3 \times C_3$. 

For (iii), note that by embedding 
$K \leq \Stab_{\Gamma} (1) \hookrightarrow \Gamma^{(\times 3)}$, 
$K / K^{(\times 3)}$ is naturally a subgroup of 
$(\Gamma/K)^{(\times 3)}$, so by (ii) 
is an elementary abelian $3$-group. 
Moreover we have that $K$ is the normal closure of $x_1$. 
Consider the action of $\Gamma$ on $K / K^{(\times 3)}$ 
by conjugation. 
Since $b$ acts trivially, $K / K^{(\times 3)}$ 
is generated by the images of $x_1$, $x_1 ^a$ and $x_1 ^{a^2}$. 

Now $x_1 , x_1 ^a$ are non-zero and independent modulo 
$K^{(\times 3)}$ ($x_1 ^a$ has non-zero $a$-component in 
the $3$rd co-ordinate, which $x_1$ does not, for instance). 
However $x_1 ^{a^2} \equiv (x_1 ^{a} x_1)^{-1} \mod K^{(\times 3)}$. 
Hence $K / K^{(\times 3)} \cong C_3 \times C_3$ 
is spanned by $x_1$ and $x_1 ^a$, 
or equivalently by $x_1$ and $(x_1 ^a)^{-1} x_1 = x_2$. 
\end{proof}

Thus we have a descending sequence of finite-index normal subgroups: 
\begin{center}
$\Gamma \geq K \geq K^{(\times 3)} \geq K^{(\times 9)} 
\geq \ldots \geq K^{(\times 3^m)} \geq \ldots$
\end{center}
with $\lvert \Gamma : K^{(\times 3^m)} \rvert = 3^{3^m + 1}$. 
Moreover $K^{(\times 3^m)} \leq \Stab_{\Gamma} (m+1)$ 
for all $m \in \mathbb{N}$. 
Since, by \cite{BartGrig} Proposition 6.5: 
\begin{equation} \label{FabGupStabindex}
\lvert \Gamma : \Stab_{\Gamma} (m+1) \rvert = 3^{3^m + 1}
\end{equation}
we conclude the following. 

\begin{coroll} \label{KStabcoroll}
$K^{(\times 3^m)} = \Stab_{\Gamma} (m+1)$ for all $m \geq 1$. 
\end{coroll}

We introduce some new notation. 
For $x \in \Gamma$, let 
$\mathbf{0}(x),\mathbf{1}(x),\mathbf{2}(x) 
\in \Gamma^{(\times 3)}$ be given by: 
\begin{center}
$\mathbf{0}(x) = (x,1,1)$, 
$\mathbf{1}(x) = (x,x^{-1},1)$, 
$\mathbf{2}(x) = (x,x^{-2},x)$
\end{center}
so that $\mathbf{1}(x)=[a,\mathbf{0}(x)]$ 
and $\mathbf{2}(x)=[a,\mathbf{1}(x)]$. 
Define the subgroups:
\begin{center} 
$L = \langle x_2 , K^{(\times 3)} \rangle$, 
$K_{\mathbf{1}} ^{(\times 3)} 
= \langle \mathbf{1}(x_1),\mathbf{2}(x_1),L^{(\times 3)} \rangle$, 
$K_{\mathbf{2}} ^{(\times 3)} 
= \langle \mathbf{2}(x_1),L^{(\times 3)} \rangle$
\end{center}
(further writing $K_{\mathbf{0}} ^{(\times 3)} 
= K^{(\times 3)}$)
and for $i \geq 2$, $r \in \lbrace 0,1,2 \rbrace$ 
define recursively: 
\begin{center} 
$K_{\mathbf{r}} ^{(\times 3^i)} 
= (K_{\mathbf{r}} ^{(\times 3)})^{(\times 3^{i-1})}$. 
\end{center}
We therefore have, for each $i \geq 1$, 
a descending chain of subgroups: 
\begin{center}
$K^{(\times 3^{i+1})} \leq L^{(\times 3^i)} 
\leq K_{\mathbf{2}} ^{(\times 3^i)} 
\leq K_{\mathbf{1}} ^{(\times 3^i)} 
\leq K^{(\times 3^i)}$. 
\end{center}
Further define, for $r,s \in \lbrace 0,1,2 \rbrace$, the subgroup: 
\begin{center}
$K_{\mathbf{rs}} ^{(\times 9)} 
= \langle L^{(\times 9)} \cup \lbrace \mathbf{tu}(x_1) 
: t+3u \geq r+3s \rbrace \rangle$
\end{center}
and for $i \geq 3$ define recursively: 
$K_{\mathbf{rs}} ^{(\times 3^i)} = 
(K_{\mathbf{rs}} ^{(\times 9)})^{(\times 3^{i-2})}$. 
Thus for $i \geq 2$, 
\begin{center}
$K_{\mathbf{1}} ^{(\times 3^i)} 
\leq K_{\mathbf{20}} ^{(\times 3^i)}
\leq K_{\mathbf{10}} ^{(\times 3^i)}
\leq K_{\mathbf{00}} ^{(\times 3^i)} = K^{(\times 3^i)}$; \\
$K_{\mathbf{2}} ^{(\times 3^i)} 
\leq K_{\mathbf{21}} ^{(\times 3^i)}
\leq K_{\mathbf{11}} ^{(\times 3^i)}
\leq K_{\mathbf{01}} ^{(\times 3^i)}
= K_{\mathbf{1}} ^{(\times 3^i)}$; \\
$L^{(\times 3^i)}  
\leq K_{\mathbf{22}} ^{(\times 3^i)}
\leq K_{\mathbf{12}} ^{(\times 3^i)}
\leq K_{\mathbf{02}} ^{(\times 3^i)}
= K_{\mathbf{2}} ^{(\times 3^i)}$. 
\end{center}

We stress that the symbols $K_{\mathbf{r}}$ and $K_{\mathbf{rs}}$ 
by themselves have no meaning, 
so that $K_{\mathbf{r}} ^{(\times 3^i)}$ 
and $K_{\mathbf{rs}} ^{(\times 3^i)}$ 
are not $3^i$-fold direct products of groups in any natural way. 
It would be more proper to write these groups as 
$( K^{(\times 3^i)} )_{\mathbf{r}}$ 
and $( K^{(\times 3^i)} )_{\mathbf{rs}}$; 
we only refrain from doing so for reasons of easy readability. 

\begin{lem}
The following subgroups of $\Gamma$ are normal for all $i \geq 0$
 and $r,s \in \lbrace 0,1,2 \rbrace$. 
\begin{itemize}
\item[(i)] $L^{(\times 3^i)}$; 
\item[(ii)] $K_{\mathbf{r}} ^{(\times 3^{i+1})}$; 
\item[(iii)] $K_{\mathbf{rs}} ^{(\times 3^{i+2})}$. 
\end{itemize}
\end{lem}

\begin{proof} 
Since $\Gamma$ is generated by $a$ and $b$ 
it suffices to check that each subgroup is preserved 
under conjugation by these two elements. 

Let $H$ be one of $L$, $K_{\mathbf{r}} ^{(\times 3)}$ or  
$K_{\mathbf{rs}} ^{(\times 9)}$. 
We first observe that the normality of $H^{(\times 3^i)}$ 
for all $i \geq 1$ follows from that of $H$. 
For suppose by induction that $H^{(\times 3^i)}$ 
is normal in $\Gamma$ for smaller $i$. 
We have $H^{(\times 3^i)} = (H^{(\times 3^{i-1})})^{(\times 3)}$; 
conjugation by $a$ acts by permuting 
these $H^{(\times 3^{i-1})}$-factors, 
and conjugation by $b$ acts on each 
$H^{(\times 3^{i-1})}$-factor 
by conjugation by $a$, $b$ or $1$. 

Next, recall that $K^{(\times 3)} 
= \Stab_{\Gamma} (2) \vartriangleleft \Gamma$ 
by Corollary \ref{KStabcoroll}. 
$L/K^{(\times 3)} \cong C_3$ is generated by $x_2$, 
and direct calculation yields 
$[a,x_2],[b,x_2] \in K^{(\times 3)}$, 
from which normality of $L$, and hence (i), follows. 

For (ii), note that 
$K_{\mathbf{2}} ^{(\times 3)}/L^{(\times 3)} \cong C_3$ 
is generated by $\mathbf{2}(x_1)$. 
We calculate $[a,\mathbf{2}(x_1)],[b,\mathbf{2}(x_1)]
\in L^{(\times 3)}$, whence 
$K_{\mathbf{2}} ^{(\times 3)} \vartriangleleft \Gamma$. 
Similarly 
$K_{\mathbf{1}} ^{(\times 3)}/K_{\mathbf{2}} ^{(\times 3)} 
\cong C_3$ is generated by $\mathbf{1}(x_1)$. 
We calculate $[a,\mathbf{1}(x_1)],[b,\mathbf{1}(x_1)]
\in K_{\mathbf{2}} ^{(\times 3)}$, whence 
$K_{\mathbf{1}} ^{(\times 3)} \vartriangleleft \Gamma$. 

Finally for (iii), note that 
$K_{\mathbf{20}} ^{(\times 9)} / K_{\mathbf{1}} ^{(\times 9)} ,
K_{\mathbf{21}} ^{(\times 9)} / K_{\mathbf{2}} ^{(\times 9)}, 
K_{\mathbf{22}} ^{(\times 9)} / L^{(\times 9)} \cong C_3$ 
are generated by 
$\mathbf{20}(x_1),\mathbf{21}(x_1)$ and $\mathbf{22}(x_1)$, 
respectively. We compute: 
\begin{align*}
[a,\mathbf{20}(x_1)],[b,\mathbf{20}(x_1)] 
& \in K_{\mathbf{1}} ^{(\times 9)}; \\
[a,\mathbf{21}(x_1)],[b,\mathbf{21}(x_1)] 
& \in K_{\mathbf{2}} ^{(\times 9)}; \\
[a,\mathbf{22}(x_1)],[b,\mathbf{22}(x_1)] 
& \in L^{(\times 9)} 
\end{align*}
so that $K_{\mathbf{20}} ^{(\times 9)},
K_{\mathbf{21}} ^{(\times 9)},
K_{\mathbf{22}} ^{(\times 9)} \vartriangleleft \Gamma$. 
Meanwhile for $r=0,1$ and $s=0,1,2$, 
$K_{\mathbf{rs}} ^{(\times 9)}/K_{\mathbf{(r+1)s}} ^{(\times 9)} 
\cong C_3$ is generated by $\mathbf{rs}(x_1)$. 
We compute $[a,\mathbf{rs}(x_1)],[b,\mathbf{rs}(x_1)] 
\in K_{\mathbf{(r+1)s}} ^{(\times 9)}$, 
so that the normality of $K_{\mathbf{rs}} ^{(\times 9)}$ 
follows from that of $K_{\mathbf{(r+1)s}} ^{(\times 9)}$. 
\end{proof}

\begin{lem} \label{FabGupcommlem}
We have the following inclusions of subgroups. 
\begin{itemize}
\item[(i)] $[K,K^{(\times 27)}] \leq 
K_{\mathbf{10}} ^{(\times 27)}$; 
\item[(ii)] $[K,K_{\mathbf{10}} ^{(\times 27)}] \leq 
K_{\mathbf{20}} ^{(\times 27)}$; 
\item[(iii)] $[K^{(\times 3)},K_{\mathbf{20}} ^{(\times 27)}] 
\leq K_{\mathbf{1}} ^{(\times 27)}$; 
\item[(iv)] $[K^{(\times 3)},K_{\mathbf{1}} ^{(\times 27)}] \leq 
K_{\mathbf{2}} ^{(\times 27)}$; 
\item[(v)] $[K^{(\times 3)},K_{\mathbf{2}} ^{(\times 27)}] \leq 
L^{(\times 27)}$; 
\item[(vi)] $[K^{(\times 9)},L^{(\times 27)}] \leq K^{(\times 81)}$. 
\end{itemize}
\end{lem}

\begin{proof}
For (i) we have: 
\begin{center}
$[K,K^{(\times 27)}] 
\leq [\Gamma,K^{(\times 9)}]^{(\times 3)}$
\end{center}
so it suffices to show that 
$[\Gamma,K^{(\times 9)}] \leq K_{\mathbf{10}} ^{(\times 9)}$. 
This holds, since 
$K^{(\times 9)} / K_{\mathbf{10}} ^{(\times 9)} \cong C_3$ 
is generated by $\mathbf{00}(x_1)$, 
and we calculate: 
\begin{center}
$[a,\mathbf{00}(x_1)] = \mathbf{10}(x_1) 
\in K_{\mathbf{10}} ^{(\times 9)}$;\\
$[b,\mathbf{00}(x_1)] = \mathbf{01}(x_1)
\in K_{\mathbf{10}} ^{(\times 9)}$. 
\end{center}
Similarly for (ii) we have: 
\begin{center}
$[K,K_{\mathbf{10}} ^{(\times 27)}] 
\leq [\Gamma,K_{\mathbf{10}} ^{(\times 9)}]^{(\times 3)}$
\end{center}
so it suffices to show that 
$[\Gamma,K_{\mathbf{10}} ^{(\times 9)}] 
\leq K_{\mathbf{20}} ^{(\times 9)}$. 
This holds, since 
$K_{\mathbf{10}} ^{(\times 9)} / K_{\mathbf{20}} ^{(\times 9)} \cong C_3$ 
is generated by $\mathbf{10}(x_1)$, 
and we calculate: 
\begin{center}
$[a,\mathbf{10}(x_1)] = \mathbf{20}(x_1)
\in K_{\mathbf{20}} ^{(\times 9)}$;\\
$[b,\mathbf{10}(x_1)] = \mathbf{01}(x_1) 
\in K_{\mathbf{20}} ^{(\times 9)}$. 
\end{center}
For (iii) we have: 
\begin{center}
$[K^{(\times 3)},K_{\mathbf{20}} ^{(\times 27)}] 
\leq [K^{(\times 3)},K^{(\times 27)}] 
\leq [\Gamma,K^{(\times 3)}]^{(\times 9)}$
\end{center}
so it suffices to show that 
$[\Gamma,K^{(\times 3)}] \leq K_{\mathbf{1}} ^{(\times 3)}$. 
$K^{(\times 3)}/K_{\mathbf{1}} ^{(\times 3)} \cong C_3$ 
is generated by $\mathbf{0}(x_1)$, 
and we calculate: 
\begin{center}
$[a,\mathbf{0}(x_1)] = \mathbf{1}(x_1)
\in K_{\mathbf{1}} ^{(\times 3)}$;\\
$[b,\mathbf{0}(x_1)] = \mathbf{0}(x_2)
\in K_{\mathbf{1}} ^{(\times 3)}$. 
\end{center}
For (iv) we have: 
\begin{center}
$[K^{(\times 3)},K_{\mathbf{1}} ^{(\times 27)}] 
\leq [\Gamma,K_{\mathbf{1}} ^{(\times 3)}]^{(\times 9)}$ 
\end{center}
so it suffices to show that 
$[\Gamma,K_{\mathbf{1}} ^{(\times 3)}] 
\leq K_{\mathbf{2}} ^{(\times 3)}$. 
This is the case, since 
$K_{\mathbf{1}} ^{(\times 3)}/K_{\mathbf{2}} ^{(\times 3)} 
\cong C_3$ is generated by $\mathbf{1}(x_1)$, 
and we may calculate: 
\begin{center} 
$[a,\mathbf{1}(x_1)] = \mathbf{2}(x_1)
\in K_{\mathbf{2}} ^{(\times 3)}$;\\
$[b,\mathbf{1}(x_1)] = \mathbf{0}(x_2)
\in K_{\mathbf{2}} ^{(\times 3)}$. 
\end{center}
For (v) we have: 
\begin{center}
$[K^{(\times 3)},K_{\mathbf{2}} ^{(\times 27)}] \leq 
[\Gamma,K_{\mathbf{2}} ^{(\times 3)}]^{(\times 9)}$ 
\end{center}
so it suffices to show that 
$[\Gamma,K_{\mathbf{2}} ^{(\times 3)}] 
\leq L^{(\times 3)}$. 
This is the case, since 
$K_{\mathbf{2}} ^{(\times 3)}/L^{(\times 3)} 
\cong C_3$ is generated by $\mathbf{2}(x_1)$, 
and we may calculate: 
\begin{center}
$[a,\mathbf{2}(x_1)] = (1,x_1 ^{-3},x_1 ^3)
\in K^{(\times 9)} \subseteq L^{(\times 3)}$ 
(by Proposition \ref{branchingpropn} (iii));\\
$[b,\mathbf{2}(x_1)] = (x_2,1,[b,x_1])
\in L^{(\times 3)}$ (since $[b,x_1] \in K^{(\times 3)}$). 
\end{center}
Finally for (iv) we have: 
\begin{center}
$[K^{(\times 9)},L^{(\times 27)}] \leq [\Gamma,L]^{(\times 27)}$ 
\end{center}
so it suffices to check that $[\Gamma,L] \leq K^{(\times 3)}$. 
This is so because $L/K^{(\times 3)} \cong C_3$ 
is generated by $x_2$ and we calculate: 
\begin{center} 
$[a,x_2] = \big(x_1 ^{-1},(x_1 ^{-1})^a,(x_1 ^{-1})^{a^{-1}} \big)
\in K^{(\times 3)}$;\\
$[b,x_2] \in K^{(\times 3)}$. 
\end{center}
\end{proof}

We now construct our approximations, by products of cubes, 
to elements lying deeper in our chain of subgroups. 

\begin{propn} \label{cubeexprnlem}
Let $i \geq 0$. 
\begin{itemize}
\item[(i)] For all $z \in K^{(\times 3^{i+3})}$, 
there exist $y_1 , \ldots , y_9 \in K^{(\times 3^i)}$ 
such that:
\begin{equation}
z \equiv \prod_{j=1} ^{27} y_i ^3 
\mod K_{\mathbf{10}} ^{(\times 3^{i+3})}
\end{equation}

\item[(ii)] For all $z \in K_{\mathbf{10}} ^{(\times 3^{i+3})}$, 
there exist $y_1 , \ldots , y_{18} \in K^{(\times 3^i)}$ 
such that:
\begin{equation}
z \equiv \prod_{j=1} ^{27} y_i ^3 
\mod K_{\mathbf{20}} ^{(\times 3^{i+3})}
\end{equation}

\item[(iii)] For all $z \in K_{\mathbf{20}} ^{(\times 3^{i+3})}$, 
there exist $y_1 , \ldots , y_{4} \in K^{(\times 3^{i+1})}$
such that:
\begin{equation}
z \equiv \prod_{j=1} ^{4} y_i ^3 
\mod K_{\mathbf{1}} ^{(\times 3^{i+3})}
\end{equation}

\item[(iv)] For all $z \in K_{\mathbf{1}} ^{(\times 3^{i+3})}$, 
there exist $y_1 , \ldots , y_{6} \in K^{(\times 3^{i+1})}$
such that:
\begin{equation}
z \equiv \prod_{j=1} ^{6} y_i ^3 
\mod K_{\mathbf{2}} ^{(\times 3^{i+3})}
\end{equation}

\item[(v)] For all $z \in K_{\mathbf{2}} ^{(\times 3^{i+3})}$, 
there exist $y_1 , \ldots , y_{6} \in K^{(\times 3^{i+1})}$
such that:
\begin{equation}
z \equiv \prod_{j=1} ^{6} y_i ^3 
\mod L ^{(\times 3^{i+3})}
\end{equation}

\item[(vi)] For all $z \in L^{(\times 3^{i+3})}$, 
there exist $y_1 , y_2 , y_{3} \in K^{(\times 3^{i+2})}$
such that:
\begin{equation}
z \equiv \prod_{j=1} ^{3} y_i ^3 
\mod K^{(\times 3^{i+4})}
\end{equation}

\end{itemize}
\end{propn}

\begin{lem} \label{cubecomplem}
The following identities hold in $\Gamma$. 
\begin{itemize}
\item[(i)] $\mathbf{000}(x_1) 
= \big( x_1 ^{ba} \mathbf{0}(x_1)^{b^a} \big)^3 \big( x_1 ^{ba} \big)^{-3} \big( \mathbf{0}(x_1)^{b^a} \big)^{-3}$; 

\item[(ii)] $\mathbf{010}(x_1)
= \big( \mathbf{0}(x_1)^{b^a b} \big)^3
\big( x_1 ^{bab} \big)^{3} 
\big( x_1 ^{bab} \mathbf{0}(x_1)^{b^a b} \big)^{-3} 
\big( x_1 ^{ba} \mathbf{0}(x_1)^{b^a} \big)^3 
\big( x_1 ^{ba} \big)^{-3} 
\big( \mathbf{0}(x_1)^{b^a} \big)^{-3}$

\item[(iii)] $\mathbf{20}(x_1) \equiv 
x_2 ^{-3} \big( x_1 ^{ba^{-1}} \big)^{-3} 
\big( x_1 ^{a} \big)^{-3} \big( x_1 ^{b^{-1}} \big)^{-3} 
\mod K_{\mathbf{1}} ^{(\times 9)}$; 

\item[(iv)] $\mathbf{01}(x_1) = x_1 ^3 (x_1 ^b)^{-3}$; 

\item[(v)] $\mathbf{02}(x_1)\equiv 
\big( \mathbf{0}(x_1)(x_1)^{a^{-1}} \big)^3 
\big( x_1 ^{a^{-1}} \big)^{-3}
\mod L^{(\times 9)}$; 

\item[(vi)] $\mathbf{0}(x_2)\equiv x_1 ^{-3} \mod K^{(\times 9)}$. 

\end{itemize}
\end{lem}

\begin{proof}
All these approximations are achieved by direct computation. 
We work through (i) and (ii) in detail and leave the others 
(which are easier) as an exercise to the reader. 

For (i), recall that $b = (a,1,b)$ and $x_1 = (b^{-1}a,a^{-1},b)$, 
so $b^a = (b,a,1)$ and: 
\begin{equation*}
x_1 ^{ba} = (a^{-1}b^{-1}a^{-1},a^{-1},b)^a 
= (b,a^{-1}b^{-1}a^{-1},a^{-1})
\end{equation*}
\begin{equation*}
\mathbf{0}(x_1)^{b^a} = \mathbf{0}(x_1 ^b) 
\end{equation*}
\begin{equation*}
x_1 ^{ba} \mathbf{0}(x_1)^{b^a} = (x_1 b,a^{-1}b^{-1}a^{-1},a^{-1})
\end{equation*}
Thus: 
\begin{equation*}
(x_1 ^{ba})^{-3} = (1,(aba)^3,1)
\end{equation*}
\begin{equation*}
(x_1 ^{ba} \mathbf{0}(x_1)^{b^a})^3 = ((x_1 b)^3,(aba)^{-3},1)
\end{equation*}
so: 
\begin{equation} \label{000comp1}
\big( x_1 ^{ba} \mathbf{0}(x_1)^{b^a} \big)^3 \big( x_1 ^{ba} \big) ^{-3} \big( \mathbf{0}(x_1)^{b^a} \big)^{-3} 
= \mathbf{0}\big( (x_1 b)^3 (x_1 ^b)^{-3} \big)
\end{equation}
and: 
\begin{equation*}
x_1 b = (b^{-1}a^{-1},a^{-1},b^{-1})
\end{equation*}
\begin{equation*}
x_1 ^b = (a^{-1}b^{-1}a^{-1},a^{-1},b)
\end{equation*}
hence: 
\begin{equation} \label{000comp2}
(x_1 b)^3 (x_1 ^b)^{-3} 
= \mathbf{0}\big( (b^{-1}a^{-1})^3 (aba)^3 \big)
\end{equation}
while: 
\begin{equation} \label{000comp3}
(b^{-1}a^{-1})^3 (aba)^3
= b^{-1}(b^{-1})^a b(b^a) = \mathbf{0}(x_1)\text{.}
\end{equation}
Combining (\ref{000comp1}), (\ref{000comp2}) and (\ref{000comp3}), 
we have the required conclusion.  

(ii) now follows from (i), noting that: 
\begin{center}
$\mathbf{010}(x_1) = \mathbf{0}([a,\mathbf{00}(x_1)])
 = \mathbf{0}(\mathbf{00}(x_1)^a)^{-1}  \mathbf{000}(x_1)
 = (\mathbf{000}(x_1)^b)^{-1} \mathbf{000}(x_1)$. 
\end{center}
\end{proof}

\begin{proof}[Proof of Proposition \ref{cubeexprnlem}]
In each of (i)-(vi), we have normal subgroups $M\geq N\geq N_{\ast}$ 
and our claim is that for the relevant $A \in \mathbb{N}$,  
for all $i \geq 0$ and all $z \in N^{(\times 3^i)}$, 
there exists $y_1 , \ldots , y_A \in M^{(\times 3^i)}$ 
such that $z \equiv y_1 ^3 \cdots y_A ^3 \mod N_{\ast} ^{(\times 3^i)}$. 
We first note that it suffices to prove the claim for $i=0$. 
For if we write: 
\begin{center}
$z = (z^{(j)})_{j=1} ^{3^i}$ with $z^{(j)} \in N$, 
\end{center}
we have $y_1 ^{(j)} , \ldots , y_A ^{(j)} \in M$ such that 
$(y_1 ^{(j)})^3 \cdots (y_A ^{(j)})^3 \equiv z^{(j)} \mod N_{\ast}$
(by the claim with $i=0$). 
Then, setting: 
\begin{center}
$y_1 = (y_1 ^{(j)})_{j=1} ^{3^i} , \ldots , 
y_A = (y_A ^{(j)})_{j=1} ^{3^i}$
\end{center}
we have the required conclusion. 
\begin{itemize}
\item[(i)] By Lemma \ref{cubecomplem} (i), 
there exist $u,v,w \in K$ such that:
\begin{center} 
$\mathbf{000}(x_1) = u^3 v^3 w^3$. 
\end{center}
For all $z \in K^{(\times 27)}$, 
there exist $\lambda,\mu,\nu \in \lbrace 0,\pm 1 \rbrace$ 
such that: 
\begin{center}
$z \equiv \mathbf{000}(x_1)^{\lambda} 
\big( \mathbf{000}(x_1)^a \big)^{\mu} 
\big( \mathbf{000}(x_1)^{a^2} \big)^{\nu} 
\mod K^{(\times 27)} _{\mathbf{10}}$. 
\end{center} 
Thus $z^{\prime} = (u^3 v^3 w^3)^{\lambda}
((u^a)^3 (v^a)^3 (w^a)^3)^{\mu}
((u^{a^2})^3 (v^{a^2})^3 (w^{a^2})^3)^{\nu}$ 
is a product of nine cubes in $K$ and 
$z \equiv z^{\prime} \mod K^{(\times 27)} _{\mathbf{10}}$. 

\item[(ii)] For all $z \in K^{(\times 27)} _{\mathbf{10}}$, 
there exist $\lambda,\mu,\nu \in \lbrace 0,\pm 1 \rbrace$ 
such that: 
\begin{center}
$z \equiv \mathbf{010}(x_1)^{\lambda} 
\big( \mathbf{010}(x_1)^a \big)^{\mu} 
\big( \mathbf{010}(x_1)^{a^2} \big)^{\nu} 
\mod K^{(\times 27)} _{\mathbf{20}}$. 
\end{center} 
By Lemma \ref{cubecomplem} (ii), 
$\mathbf{010}(x_1)$ is the product of six cubes in $K$, 
so that (modulo $K^{(\times 27)} _{\mathbf{20}}$) 
$z$ is a product of eighteen cubes in $K$. 

\item[(iii)] Arguing as in the first paragraph of this proof, 
it suffices to show that for all 
$z \in K_{\mathbf{20}} ^{(\times 9)}$
there exist $y_1 , \ldots , y_4 \in K$ such that: 
\begin{center}
$z \equiv y_1 ^3 y_2 ^3 y_3 ^3 y_4 ^3 
\mod K_{\mathbf{1}} ^{(\times 9)}$. 
\end{center}
By Lemma \ref{cubecomplem} (iii), 
there exist $t,u,v,w \in K$ such that: 
\begin{center}
$\mathbf{20}(x_1) \equiv t^3 u^3 v^3 w^3 
\mod K_{\mathbf{1}} ^{(\times 9)}$. 
\end{center}
For all $z \in K_{\mathbf{20}} ^{(\times 9)}$ 
there exists $\lambda \in \lbrace 0,\pm 1 \rbrace$ 
such that: 
\begin{center}
$z \equiv \mathbf{20}(x_1)^{\lambda} 
\mod K_{\mathbf{1}} ^{(\times 9)}$
\end{center}
and the claim follows. 

\item[(iv)] Arguing as in the first paragraph of this proof, 
it suffices to show that for all 
$z \in K_{\mathbf{1}} ^{(\times 9)}$ 
there exist $y_1 , \ldots , y_6 \in K$ such that: 
\begin{center}
$z \equiv y_1 ^3 \cdots y_6 ^3 
\mod K_{\mathbf{2}} ^{(\times 9)}$. 
\end{center}
By Lemma \ref{cubecomplem} (iv) there exist $u,v \in K$ such that 
$\mathbf{01}(x_1) = u^3 v^3$, 
so that $\mathbf{01}(x_1)^a = (u^a)^3 (v^a)^3$ and 
$\mathbf{01}(x_1)^{a^2} = (u^{a^2})^3 (v^{a^2})^3$. 
For all $z \in K_{\mathbf{1}} ^{(\times 9)}$ 
there exist $\lambda,\mu,\nu \in \lbrace 0,\pm 1\rbrace$ 
such that: 
\begin{center}
$z \equiv \mathbf{01}(x_1)^{\lambda}  
\big( \mathbf{01}(x_1)^a \big)^{\mu} 
\big( \mathbf{01}(x_1)^{a^2} \big)^{\nu}
\mod K_{\mathbf{2}} ^{(\times 9)}$
\end{center}
and the claim follows. 

\item[(v)] As in (iii) and (iv) 
it suffices to show that for all 
$z \in K_{\mathbf{2}} ^{(\times 9)}$ there exist 
$y_1 , \ldots , y_6 \in K$ such that: 
\begin{center}
$z \equiv y_1 ^3 \cdots y_6 ^3 \mod L^{(\times 9)}$. 
\end{center}
By Lemma \ref{cubecomplem} (v) there exist $u,v \in K$ such that 
$\mathbf{02}(x_1) \equiv u^3 v^3 \mod L^{(\times 9)}$, so that: 
\begin{center}
$\mathbf{02}(x_1)^a \equiv (u^a)^3 (v^a)^3 , \mathbf{02}(x_1)^{a^2} \equiv (u^{a^2})^3 (v^{a^2})^3 \mod L^{(\times 9)}$. 
\end{center}
For all $z \in K_{\mathbf{2}} ^{(\times 9)}$ 
there exist $\lambda,\mu,\nu \in \lbrace 0,\pm 1\rbrace$ 
such that: 
\begin{center}
$z \equiv \mathbf{02}(x_1)^{\lambda}  
\big( \mathbf{02}(x_1)^a \big)^{\mu} 
\big( \mathbf{02}(x_1)^{a^2} \big)^{\nu}
\mod L^{(\times 9)}$
\end{center}
and the claim follows. 

\item[(vi)] As in (iii)-(v) it suffices to show that for all 
$z \in L^{(\times 3)}$ there exist $y_1,y_2,y_3 \in K$ 
such that: 
\begin{center}
$z \equiv y_1 ^3 y_2 ^3 y_3 ^3 \mod K^{(\times 9)}$. 
\end{center}
By Lemma \ref{cubecomplem} (vi) there exists $u \in K$ such that 
$\mathbf{0}(x_2) \equiv u^3 \mod K^{(\times 9)}$. 
For all $z \in L^{(\times 3)}$ there exist 
$\lambda,\mu,\nu \in \lbrace 0,\pm 1\rbrace$ such that: 
\begin{center}
$z \equiv \mathbf{0}(x_2)^{\lambda} 
\big( \mathbf{0}(x_2)^a \big)^{\mu} 
\big( \mathbf{0}(x_2)^{a^2} \big)^{\nu} 
\equiv (u^{\lambda})^3 
\big( (u^{\lambda})^a \big)^3 
\big( (u^{\lambda})^{a^2} \big)^3 \mod K^{(\times 9)}$
\end{center}
as required. 
\end{itemize}
\end{proof}

\begin{rmrk} \label{FabGupapproxtimermrk}
Note that the proof of Proposition \ref{cubeexprnlem} 
in each of the cases (i)-(vi) explicitly constructs 
the elements $y_j$, the product of whose cubes approximates $z$. 
Indeed, the direct product decomposition 
(and corresponding reduction to the case $i=0$) 
achieved in the first paragraph of the proof 
provides an algorithm which given $z$, 
constructs the $y_j$ in time $O(3^i)$ 
(since, having reduced to $i=0$, 
all computations take place in $K/K^{(\times 81)}$, 
so represent a bounded problem). 
\end{rmrk}

\begin{proof}[Proof of Theorem \ref{FabGupMainThm}]
This will follow from Theorem \ref{PotentSKP}. 
We set $k_n = 3$ and define the sequences 
$(M_n)_n$, $(N_n)_n$, $(A_n)_n$ as follows. 
Write $n = 6q+r$, with $1 \leq r \leq 6$. Then: 
\begin{align*}
M_{6q+1} &= K^{(\times 3^{q})}; 
&N_{6q+1} &= K^{(\times 3^{q+3})}; 
&A_{6q+1} &= 9; \\
M_{6q+2} &= K^{(\times 3^{q})}; 
&N_{6q+2} &= K_{\mathbf{10}} ^{(\times 3^{q+3})}; 
&A_{6q+2} &= 18; \\
M_{6q+3} &= K^{(\times 3^{q+1})}; 
&N_{6q+3} &= K_{\mathbf{20}} ^{(\times 3^{q+3})}; 
&A_{6q+3} &= 4; \\
M_{6q+4} &= K^{(\times 3^{q+1})}; 
&N_{6q+4} &= K_{\mathbf{1}} ^{(\times 3^{q+3})}; 
&A_{6q+4} &= 6; \\
M_{6q+5} &= K^{(\times 3^{q+1})}; 
&N_{6q+5} &= K_{\mathbf{2}} ^{(\times 3^{q+3})}; 
&A_{6q+5} &= 6; \\
M_{6q+6} &= K^{(\times 3^{q+2})}; 
&N_{6q+6} &= L^{(\times 3^{q+3})}; 
&A_{6q+6} &= 3 \\
\end{align*}
Hypothesis (i) of Theorem \ref{PotentSKP} is clear 
and hypothesis (ii) follows immediately 
from Lemma \ref{FabGupcommlem}. 
Hypothesis (iii) follows from the fact that $(N_n)_n$ 
is a descending sequence and that: 
\begin{center}
$N_{6q+1}/N_{6(q+1)+1}
=K^{(\times 3^{q+3})}/K^{(\times 3^{q+4})}
\cong C_3 ^{(\times 3^{q+3})}$
\end{center} 
has exponent $3$. 
Hypothesis (iv) is precisely the content of 
Proposition \ref{cubeexprnlem}. 

Let $\tilde{C} = (3A_1+1)\cdots (3A_6+1)= 72272200$. 
We conclude, by (\ref{FabGupStabindex}) and 
Corollary \ref{KStabcoroll}, that for $m \geq 4$, 
\begin{align*}
\diam^+ (\Gamma/\Stab_{\Gamma}(m)) 
& = \diam^+ (\Gamma/K^{(\times 3^{m-1})}) \\
& = \diam^+ (\Gamma/N_{6(m-4)+1}) \\
& \leq \lvert \Gamma/\Stab_{\Gamma}(4) \rvert 
\tilde{C}^{m-4}\\
& = (3^{26}/\tilde{C}^3)
(27/26\log(3))^{\frac{\log(\tilde{C})}{\log(3)}}
\log \lvert \Gamma:\Stab_{\Gamma}(m) \rvert^{\frac{\log(\tilde{C})}{\log(3)}}\text{.}
\end{align*}
For the directed navigation problem, 
we may take $f(n) = O(3^{n/6})$. 
For, once again writing $n=6q+r$, 
$\Gamma / N_n$ is a quotient of 
$\Gamma / N_{6(q+1)+1} = \Gamma/K^{(\times 3^{q+4})}
= \Gamma/\Stab_{\Gamma}(q+5)$, 
a permutation group of degree $3^{q+5}$, 
so products and inverses in $\Gamma / N_n$ 
may be computed in time $O(3^{q}) = O(3^{n/6})$. 
The approximations required by hypothesis (c) 
of Theorem \ref{PotentSKP} may also be computed in time 
$O(3^{n/6})$, by Remark \ref{FabGupapproxtimermrk}. 
Let $\tilde{C}^{\prime} = (A_1 + 1)\cdots (A_6+1) = 186200$. 
By (\ref{runtimebound}) the algorithm for 
$\Gamma/\Stab_{\Gamma}(m) = \Gamma / N_{6(m-4)+1}$ 
runs in time:
\begin{align*}
O( 3^m (\tilde{C}^{\prime})^m \lvert S \rvert^{1+9^{14}})
& = O(\lvert S \rvert^{1+9^{14}} \log \lvert \Gamma:\Stab_{\Gamma}(m) \rvert^{1+\frac{\log(\tilde{C}^{\prime})}{\log(3)}})\text{.}
\end{align*}

\end{proof}

\section{$p$-adic Analytic Groups} \label{padicsect}

In this Section we prove a directed diameter bound 
for a sequence of quotients of an arbitrary compact $p$-adic 
group, and observe that our bound is an instance 
of the potent SKP. 
We assume that this bound is well-known, 
but we are not aware of an existing reference. 
Before stating the result we require some 
background on $p$-adic analytic groups. 
Our exposition here is based on \cite{DiDuMaSe}. 

\begin{defn}
Let $\Gamma$ be a finitely generated pro-$p$ group. 
$\Gamma$ is \emph{powerful} if $\Gamma / \overline{\mho_{p^e} (\Gamma)}$ 
is abelian, where $e=2$ when $p=2$ and $e=1$ when $p$ is odd. 
$\Gamma$ is \emph{uniform} if it is powerful and torsion-free. 
The \emph{rank} of a uniform group is the minimal 
size of a topological generating set. 
\end{defn}

There are many characterizations of $p$-adic analytic groups. 
For compact groups, perhaps the easiest to visualize is this: 
a compact topological group $\Gamma$ is $p$-adic analytic iff 
it is isomorphic to a closed subgroup of some 
$\SL_n (\hat{\mathbb{Z}}_p)$. 
Equivalently, $\Gamma$ is $p$-adic analytic iff 
it has the structure of a $p$-adic analytic manifold, 
such that the group operations are analytic functions. 
The \emph{dimension} of $\Gamma$ in this case 
is its dimension as a $p$-adic analytic manifold. 

\begin{thm}
Let $\Gamma$ be a compact $p$-adic analytic group. 
Then $\Gamma$ has an open characteristic powerful pro-$p$ 
subgroup $H$. 
\end{thm}

Let $H$ be a finitely generated powerful pro-$p$ group. 
Let $(H_i)_i$ be the lower central $p$-series of $H$. 

\begin{lem} \label{padicshrinklem}
For all $i,j$, 
\begin{itemize}
\item[(i)] $[H_i,H_j] \leq H_{i+j}$; 
\item[(ii)] $\mho_p (H_i) \leq H_{i+1}$. 
\end{itemize}
\end{lem}

\begin{thm}
$H_i$ is uniform for all sufficiently large $i$. 
In particular, every compact $p$-adic analytic group 
has an open characteristic uniform pro-$p$ subgroup. 
\end{thm}

\begin{lem} \label{padicfollowlem}
For all $i,j$, $(H_{i+1})_{j+1} = H_{i+j+1}$. 
\end{lem}

\begin{lem} \label{padicepilem}
For all $i,j$, the map $x \mapsto x^{p^j}$ 
induces an epimorphism 
$H_i/H_{i+1} \rightarrow H_{i+j}/H_{i+j+1}$. 
\end{lem}

\begin{thm}
Let $\Gamma$ be a compact $p$-adic analytic group of dimension 
$d$. Let $K$ be an open uniform subgroup of $\Gamma$. 
Then $K$ has rank $d$. 
\end{thm}

\begin{lem} \label{padiclevellem}
Let $K$ be a uniform pro-$p$ group of rank $d$. 
Let $(K_i)_i$ be the lower central $p$-series of $K$. 
Then $K_i / K_{i+1} \cong C_p ^d$. 
\end{lem}

We are now ready to state and prove our diameter bound. 

\begin{thm} \label{padicdiamthm}
Let $\Gamma$ be a compact $p$-adic analytic group of dimension 
$d$. Let $H$ be an open characteristic powerful pro-$p$ 
subgroup. Let $(H_i)_i$ be the lower central $p$-series of $H$. 
Then for all $n$, 
\begin{center}
$\diam^+ (\Gamma/H_n) \leq \lvert \Gamma:H_2 \rvert (p^{n-1}-1)/(p-1)  = O_{\Gamma} \big( \lvert \Gamma:H_n \rvert^{1/d} \big)$. 
\end{center}
\end{thm}

\begin{proof}
Let $S \subseteq \Gamma/H_n$ be a generating set. 
Then $H_1/H_2 \subseteq B^+ _S (\lvert \Gamma:H_2 \rvert)H_2/H_2$, 
so by Lemma \ref{padicepilem}, 
$H_i/H_{i+1} \subseteq B^+ _S (p^{i-1}\lvert \Gamma:H_2 \rvert)$ 
for all $i \leq n-1$, 
\begin{center}
$\diam^+ (\Gamma/H_n,S) \leq \lvert \Gamma:H_2 \rvert (1+p+\cdots +p^{n-2})$. 
\end{center}

We may also interpret this bound 
as an instance of the of the potent SKP. 
We apply Theorem \ref{PotentSKP} with 
$M_i = H_i$, $N_i = H_{i+1}$, $A_i = 1$ and $k_i = 1$. 
Hypothesis (i) of Theorem \ref{PotentSKP} is clear; 
hypotheses (ii) and (iii) follow from Lemma \ref{padicshrinklem}, 
and hypothesis (iv) follows from Lemma \ref{padicepilem}. 
Moreover $N_i = M_{i+1}$ so the improvement described 
in Remark \ref{improvementsremark} (ii) is available to us, 
with $n_0 = 1$, and the required 
bound follows from (\ref{improvermrkeqn}).

For the second equality, it suffices to note that 
$\lvert H_i : H_{i+1} \rvert \geq p^d$ for all $i$. 
This may be seen by combining Lemmas \ref{padicfollowlem}, 
\ref{padicepilem} and \ref{padiclevellem}. 
\end{proof}

\begin{rmrk}
The conclusion of Theorem \ref{padicdiamthm} is best possible 
in general: this is witnessed by the example 
$\Gamma = F \times \hat{\mathbb{Z}}_p ^d$, 
where $F$ is a finite group (which may be chosen to be of 
arbitrarily large diameter). 
Under the assumption that $\Gamma$ is \emph{Fab} 
(that is: every open subgroup has finite abelianisation) 
much stronger, indeed polylogarithmic, diameter bounds for $\Gamma/H_n$ are provided by \cite{Brad}. 
These may then be extended to the directed diameter by 
Theorem \ref{BabaiDirectedThm}. 
Nevertheless, the degree of the polylogarithmic 
upper bound for $\diam(\Gamma/H_n)$ from \cite{Brad} 
in general grows like $\log(d)$ in the dimension $d$ 
of $\Gamma$, so the conclusion of 
Theorem \ref{padicdiamthm} does improve 
upon the results of \cite{Brad} for certain 
groups $\Gamma/H_n$ when $d$ is large compared with $n$ and $p$ 
(say $\log(d) \gg \log(p) n/\log(n)$). 
\end{rmrk}

\section{Spectral Gap and Mixing Time} \label{mixsect}

Let $G$ be a finite group and $S \subseteq G$. 
Let $A_S$ be the (symmetric, normalized) adjacency operator 
on the Cayley graph $\Cay(G,S)$. 
$A_S$ is a self-adjoint operator of norm one; 
let its spectrum be: 
\begin{center}
$1 = \lambda_1 \geq \lambda_2 \geq \ldots 
\geq \lambda_{\lvert G \rvert} \geq -1$. 
\end{center}
The eigenvalue $\lambda_1$ corresponds to the 
constant functions on $G$; 
it is a simple eigenvalue iff $S$ generates $G$. 
In this case, the quantity $1-\lambda_2$ is the 
\emph{spectral gap} of the pair $(G,S)$. 

In many applications it is desirable for a Cayley graph 
to have large spectral gap. 
In particular, a family of bounded-valence Cayley graphs 
whose spectral gaps are uniformly bounded away from zero 
form an \emph{expander family}. 
There is also a close relationship between spectral gap and diameter. 

\begin{propn}[\cite{DiacSalo} Corollary 3.1]
The spectral gap of $(G,S)$ is at least 
$(2 \lvert S \rvert \diam(G,S)^2)^{-1}$. 
\end{propn}

From this inequality and our diameter bounds, 
we obtain substantial lower bounds on spectral gap 
for Cayley graphs of our groups 
(albeit weaker bounds than would be needed to verify that 
our Cayley graphs are expanders). 

A second invariant of great interest in both practical and 
theoretical contexts is the \emph{mixing time} 
of the pair $(G,S)$, which measures the time taken 
for a (symmetric) lazy random walk on $\Cay(G,S)$ 
to closely approximate the uniform distribution 
(with respect to some metric). 
Here we follow the following convention: let $\delta_e$ 
be the Dirac mass at the identity of $G$, 
and let $T_S = (A_S+I)/2$, where $I$ is the identity operator 
on $G$. 

\begin{defn}
The \emph{$\ell^{\infty}$-mixing time} of the pair $(G,S)$ 
is the smallest $l \in \mathbb{N}$ such that: 
\begin{center}
$\big\lVert T_S ^l \delta_e - \frac{1}{\lvert G \rvert} \chi_G \big\rVert_{\infty} \leq \frac{1}{2 \lvert G \rvert}$. 
\end{center}
\end{defn}

It is clear that the $\ell^{\infty}$-mixing time of $(G,S)$ 
is an upper bound for the diameter. 
Via the spectral gap, we also have a converse inequality. 

\begin{propn}[\cite{Lova} Theorem 5.1]
Suppose the pair $(G,S)$ has spectral gap $\epsilon > 0$. 
Then there exists an absolute constant $C>0$ such that the
$\ell^{\infty}$-mixing time of $(G,S)$ is at most
$(C/\epsilon) \log \lvert G \rvert$. 
\end{propn}

Using our diameter bounds, we therefore also 
obtain new upper bounds on $\ell^{\infty}$-mixing time. 

\begin{coroll}
Let $q$ be a power of $2$ and let 
$S_n \subseteq G_n = \SL_2(\mathbb{F}_q[t]/(t^n))$ 
be a generating set. Then for all $\epsilon > 0$ 
the spectral gap of $(G_n,S_n)$ is 
$\Omega_{q,\epsilon} 
\big( \lvert S_n \rvert^{-1} \log^{-C-\epsilon} \lvert G_n \rvert \big)$ 
and the $\ell^{\infty}$-mixing time of $(G_n,S_n)$ is 
$O_{q,\epsilon} 
\big( \lvert S_n \rvert \log^{1+C+\epsilon} \lvert G_n \rvert \big)$, 
where $C = 2 \log(7)/\log(4/3) \approx 13.528$. 
\end{coroll}

\begin{coroll}
Let $\Gamma$ be the Fabrykowski-Gupta group and let 
$S_n \subseteq G_n = \Gamma/\Stab_{\Gamma}(n)$ 
be a generating set. 
Then the spectral gap of $(G_n,S_n)$ 
is $\Omega \big( \lvert S_n \rvert^{-1} \log^{-C} \lvert G_n \rvert \big)$ 
and the $\ell^{\infty}$-mixing time of $(G_n,S_n)$ is 
$O \big( \lvert S_n \rvert \log^{1+C} \lvert G_n \rvert \big)$, where 
$C = 2 \log(72272200)/\log(3) \approx 32.943$. 
\end{coroll}

\begin{coroll}
Let $\Gamma$ be a compact $p$-adic analytic group 
of dimension $d$; 
let $H$ be an open characteristic powerful pro-$p$ subgroup; 
let $(H_i)_i$ be the lower central $p$-series of $H$, 
and let $S_n \subseteq G_n =  \Gamma/H_n$ be a generating set. 
Then the spectral gap of $(G_n,S_n)$ is 
$\Omega_{\Gamma} \big( \lvert S_n \rvert^{-1} \lvert G_n \rvert^{-2/d} \big)$ and the 
$\ell^{\infty}$-mixing time of $(G_n,S_n)$ is 
$O_{\Gamma} \big( \lvert S_n \rvert \lvert G_n \rvert^{2/d} \log \lvert G_n \rvert \big)$. 
\end{coroll}

\end{document}